\documentclass[12pt]{article}
\usepackage{amsmath}
\usepackage{amsfonts}
\usepackage{amssymb}
\usepackage{xcolor}
\usepackage{bigints}
\usepackage{relsize}
\newtheorem{teor}{Theorem}[section]
\newtheorem{prop}[teor]{Proposition}
\newtheorem{coro}[teor]{Corollary}
\newtheorem{lema}[teor]{Lemma}
\newtheorem{defi}{Definition}[section]
\newtheorem{ejem}{Example}[section]
\newtheorem{nota}{Remark}[section]

\allowdisplaybreaks 

\newenvironment{proof}[1][Proof]{\textbf{#1.} }{\ \rule{0.5em}{0.5em}}

\begin{document}
\def\N{\mathbb{N}}
\def\Z{\mathbb{Z}}
\def\K{\mathbb{K}}
\def\R{\mathbb{R}}
\def\D{\mathbb{D}}
\def\C{\mathbb{C}}
\def\T{\mathbb{T}}

\def\dint{\int}
\def\intc{\int_0^1}
\def\intD{\int_{\D}}
\def\intpi{\int_0^{2\pi}}
\def\sumi{\sum_{n=0}^\infty}
\def\dfrac{\frac}
\def\dsum{\sum}
\def\cuadro{\hfill $\Box$}
\def\qed{\hfill $\Box$}
\def\prueba{\vskip10pt\noindent{\it PROOF.}\hskip10pt}
\newcommand{\cuadrosymb}{\mbox{ }~\hfill~{\rule{2mm}{2mm}}}
\providecommand{\norm}[1]{\lVert#1\rVert}
\providecommand{\Norm}[1]{\left\lVert#1\right\rVert}
\providecommand{\presc}[2]{\langle #1 ,#2\rangle}
\providecommand{\Presc}[2]{\left\langle #1 ,#2\right\rangle}
\providecommand{\lcn}[2]{\mathcal{L}(#1,#2)}
\providecommand{\convsot}[1]{\xrightarrow[]{SOT}#1}
\providecommand{\convwot}[1]{\xrightarrow[]{WOT}#1}
\providecommand{\normop}[1]{\lVert#1\rVert_{{op_2}}}
\providecommand{\normb}[1]{\lVert#1\rVert_{B(\ell^2)}}
\providecommand{\normm}[1]{\lVert#1\rVert_{M(\ell^2)}}
\providecommand{\normu}[1]{\lVert#1\rVert_{U(\ell^2)}}
\providecommand{\normc}[1]{\lVert#1\rVert_{C(\ell^2)}}
\providecommand{\normbv}[1]{\lVert#1\rVert_{BV(\ell^2)}}
\providecommand{\normel}[1]{\lVert#1\rVert_{E(\ell^2)}}
\providecommand{\normcr}[1]{\lVert#1\rVert_{C_r(\ell^2)}}
\providecommand{\normbh}[1]{\lVert#1\rVert_{\Bl2}}
\providecommand{\normmh}[1]{\lVert#1\rVert_{\Ml2}}
\providecommand{\normbx}[1]{\lVert#1\rVert_{B(\ell^2(X))}}
\providecommand{\normmx}[1]{\lVert#1\rVert_{M(\ell^2(X))}}
\providecommand{\normms}[1]{\lVert#1\rVert_{ms}}
\providecommand{\Norm}[1]{\left\lVert#1\right\rVert}
\providecommand{\lk}[1]{\mathcal{L}_k^{#1}}

\providecommand{\htild}[1]{\tilde{H}^2(\mathbb{T},\mathcal{L}(H))}
\providecommand{\ltwolh}[1]{\ell^2(\mathcal{L}(H))}
\providecommand{\ltwoh}[1]{\ell^2(H)}
\providecommand{\ltwosot}[1]{\ell^2_{SOT}(\mathcal{L}(H))}
\providecommand{\linftylh}[1]{\ell^\infty(\mathcal{L}(H))}

\providecommand{\bltwoh}[1]{\mathcal{B}\left(\ell^2(H)\right)}
\def\L{{\mathcal L}}
\def\P{{\cal P}}
\def\B{\mathcal B}
\def\KO{{\cal K}}
\def\J{\mathcal J}

\def\M{\mathcal M}

\def\botimes{{\bf \otimes}}
\def\ba{\begin{eqnarray*}}
\def\ea{\end{eqnarray*}}

\def\be{\begin{equation}}

\def\ee{\end{equation}}

\def\A{{\bf A}}
\def\bB{{\bf B}}
\def\bT{{\bf T}}
\def\x{{\bf x}}
\def\y{{\bf y}}
\def\z{{\bf z}}

\def\dt{\frac{dt}{2\pi}}
\def\ds{\frac{ds}{2\pi}}

\def\d{\displaystyle}

\def\Tkj{T_{kj}}
\def\Skj{S_{kj}}
\def\Rk{{\bf R_k}}
\def\Dl{{\bf D_l}}
\def\Cj{{\bf C_j}}
\def\ej{{\bf e_j}}
\def\ek{{\bf e_k}}

\def\la{\langle}
\def\ra{\rangle}
\def\e{\varepsilon}

\def\Cl2{\mathcal C(\ell^2(H))}
\def\Pl2{\mathcal P(\ell^2(H))}
\def\Bl2{\mathcal B(\ell^2(H))}
\def\Ml2{\mathcal M_r(\ell^2(H))}
\def\l2{\ell^2(H)}

\def\sj{\sum_{j=1}^\infty}
\def\sk{\sum_{k=1}^\infty}

\def\fk{\varphi_k}
\def\fj{\varphi_j}


\def\lsot{\ell^2_{SOT}(\N,\B(H))}
\def\lsott{\ell^2_{SOT}(\N^2,\B(H))}

\title{New spaces of matrices with operator entries}
\author{Oscar Blasco, Ismael Garc\'{\i}a-Bayona\thanks{%
Partially supported by  MTM2014-53009-P(MINECO Spain) and  FPU14/01032 (MECD Spain)}}
\date{}

\maketitle

\begin{abstract} In this paper, we will consider matrices with entries in the space of operators $\mathcal{B}(H)$, where $H$ is a separable Hilbert space and consider the class of matrices that can be approached in the operator norm by matrices with a finite number of diagonals. We will use the Schur product with Toeplitz matrices generated by summability kernels to  describe such a class and show that in the case of Toeplitz matrices it can be identified with the space of continuous functions with values in $\mathcal B(H)$. We shall also introduce matriceal versions with operator entries of classical spaces of holomorphic functions such as $H^\infty(\D)$ and $A(\D)$ when dealing with upper triangular matrices.
\end{abstract}

AMS Subj. Class: Primary 47L10; 46E40, Secondary 47A56; 15B05; 46G10.


\section{Introduction. }

Throughout the paper  $(H, \|\cdot\|)$ and $\B(H)$  stand for  a separable Hilbert space and the space of bounded linear operators from $H$ into itself. We also use the notations $\ell^2(H)$ for the space of sequences ${\bf x}=(x_n)$ with  $x_n\in H$ such that $\|{\bf x}\|_{\ell^2(H)}=(\sum_{n=1}^\infty \|x_n\|^2)^{1/2}<\infty$. In the sequel we write $ \la \cdot,\cdot\ra$ and $ \ll\cdot,\cdot \gg $ for the scalar products in  $H$  and $\ell^2(H)$ respectively, that is $\ll\x,\y\gg= \sum_{n=1}^\infty \la x_n,y_n\ra$
 and we use the notation $x\ej=(0,\cdots,0,x,0,\cdots) $ for the element in $\l2$ in which $x\in H$ is placed in the $j$-th coordinate for $j\in \N$. As usual $c_{00}(H)=span\{x\ej:x\in H, j\in \N\}$.

 Given a complex Banach space $X$ we  denote  by $P(\T,X)$, $C(\T,X)$,  $L^p(\T, X)$ or $M(\T,X)$  the space of $X$-valued trigonometric polynomials, $X$-valued continuous functions,  $X$-valued strongly measurable functions with $\|f\|_{L^p(\T,X)}=(\int_0^{2\pi} \|f(e^{it})\|_X^p\dt)^{1/p}<\infty$ for $1\le p \le \infty$ (with the usual modification for $p=\infty$) and  regular $X$-valued measures of bounded variation respectively. As usual for $X=\C$ we simply write $\ell^2$, $P(\T)$, $C(\T)$, $L^p(\T)$ and $M(\T)$.

 Finally we write $\D$ for the unit disc in $\C$ and $P_a(\D,X)=\{\sum_{k=0}^N x_kz^k: x_k\in X, N\in \N\}$ for the space of analytic polynomials with values in $X$. As usual $H^\infty(\D,X)$ and $A(\D, X)$ stand for the bounded vector-valued holomorphic functions and the vector-valued disc algebra, that is the closure of analytic polynomials in $H^\infty(\D,X)$. Since we only need the basic theory, which extends to the vector-valued setting from the scalar-valued one, we refer to the books  \cite{D,G} for possible results to be used.

Of course, operators $T\in\B(\ell^2)$ can be identified with matrices whose entries are given by $\alpha_{kj}=\la T(e_j),e_k\ra$.
Given two matrices $A=(\alpha_{kj})$ and $B=(\beta_{kj})$ with complex entries, their Schur product  is defined by $A*B= (\alpha_{kj}\beta_{kj})$. This operation was shown by J. Schur \cite{Schur} to endow the space $\B(\ell^2)$  with a structure of Banach algebra, that is if $A, B\in \B(\ell^2)$  then $A*B \in \B(\ell^2)$. Moreover
\be \label{ts} \|A*B \|_{\B(\ell^2)}\le \|A \|_{\B(\ell^2)}\|B \|_{\B(\ell^2)}.\ee
The reader is referred also to  either \cite[Proposition 2.1]{Be} or \cite[Theorem 2.20]{PP} for a different proof.
More generally, a matrix $A=(\alpha_{kj})$ is said to be a Schur multiplier, to be denoted by $A\in {\mathcal M}(\ell^2)$, whenever
$A* B\in \B(\ell^2)$  for any $B\in \B(\ell^2)$. In particular, Schur's result establishes that $\B(\ell^2)\subseteq {\mathcal M}(\ell^2)$.
For the study of Schur multipliers we refer the reader to \cite{AP,Be, PP}.

Operators in $\B(\ell^2)$ and multipliers in ${\mathcal M}(\ell^2)$ are well understood for particular classes of matrices. We recall here the results for Toeplitz matrices, that is matrices with constant diagonals, $A=(\alpha_{kj})$ with  $\alpha_{kj}=\gamma_{j-k}$ for a  given sequence of complex numbers $(\gamma_l)_{l\in \Z}$. The characterization of Toeplitz matrices which define bounded operators in $\B(\ell^2)$ goes back to work of Toeplitz in \cite{T},
 who showed that  a Toeplitz matrix $A=(\alpha_{kj})\in \B(\ell^2) $ if and only if there exists $f\in L^\infty(\T) $ such that $\alpha_{kj}=\hat f(j-k), \quad k,j\in \N $. Furthermore
\be \label{tt} \|A\|_{\B(\ell^2)}=\|f\|_{L^\infty(\T)}. \ee

On the other hand, the study of Toeplitz matrices which define Schur multipliers  goes back to  Bennet \cite{Be}, who showed that
  a Toeplitz matrix $A=(\alpha_{kj})\in \M(\ell^2) $ if and only if there exists $\mu\in M(\T)$ such that $\alpha_{kj}=\hat\mu(j-k) $ for $k,j\in \N$. Furthermore
  \be \|A\|_{\mathcal M(\ell^2)}= \|\mu\|_{M(\T)}.\label{tb} \ee

The reader is also referred to  \cite{AP, BG, PP} for the proofs of the above results concerning Toeplitz matrices.

Throughout the rest of the paper we write
 $\A=(\Tkj)$ where $\Tkj\in \B(H)$ and we denote by $\Rk$, $\Cj$ and $\Dl$  for $k,j\in \N$ and $l\in \Z$ the matrices consisting on  the  $k$-row, the $j$-column and $l$-diagonal respectively, that is to say
$$\Rk=(\Tkj)_{j=1}^\infty, \quad \Cj=(\Tkj)_{k=1}^\infty, \quad \Dl=({T}_{k,k+l})_{k=1}^\infty$$

We shall use the notation  $\mathcal T$  for the  set of Toeplitz matrices, that is those matrices such that ${T}_{k,k+l}=T_l$ for $k\in \N$ and $l\in \Z$  and  $\mathcal U$ for upper triangular matrices, that is those matrices such that $\Dl=0$ for $l<0$.

We shall deal in the paper with the operator-valued version of Schur multipliers which was initiated by the authors in \cite{BB}.  Here we simply recall some definitions and refer to \cite{HNVW} for results on vector-valued Fourier analysis we may use later on.

Given a matrix $\A =(\Tkj)$ with entries $\Tkj\in \B(H)$  and $\x\in c_{00}(H)$, we  write $\A \x$ for the sequence $(\sum_{j=1}^\infty\Tkj(x_j))_k$.
We say that $\A\in \B(\ell^2(H))$ if the map $\x \to \A\x$ extends to a bounded linear operator in $\ell^2(H)$, that is there exists $C>0$ such that
$$\left(\mathlarger{\mathlarger{\sum}}_{k=1}^\infty \Norm{\sum_{j=1}^\infty \Tkj(x_j)}^2\right)^{1/2}\le C \Big(\sum_{j=1}^\infty \|x_j\|^2\Big)^{1/2}.$$

We shall write
$$\|\A\|_{\Bl2}=\inf\{C\ge 0: \|\A \x \|_{\l2}\le C \|\x\|_{\l2}\}.$$

Given two matrices $\A =(\Tkj)$ and $\bB=(\Skj)$ with entries $\Tkj, \Skj\in \B(H)$ we define the Schur product $$\A*\bB= (\Tkj\Skj)$$ where $\Tkj\Skj$ stands for  composition of the operators $\Tkj$ and $\Skj$.
Contrary to the scalar-valued case, this product is not commutative.

Given a matrix $\A =(\Tkj)$, we say that $\A$ is a right Schur multiplier (respectively left Schur multiplier), to be denoted by $\A\in \mathcal M_r(\l2)$ (respectively $\A\in \mathcal M_l(\l2)$ ), whenever $\bB* \A\in \Bl2$
 (respectively $\A* \bB\in \Bl2$ ) for any $\bB\in \Bl2$.
We shall write
$$\|\A\|_{\Ml2}=\inf\{C\ge 0: \|\bB * \A  \|_{\Bl2}\le C \|\bB\|_{\Bl2}\}$$
and
$$\|\A\|_{\mathcal M_l(\l2)}=\inf\{C\ge 0: \|\A*\bB   \|_{\Bl2}\le C \|\bB\|_{\Bl2}\}.$$
We say that $\A$ is a Schur multiplier whenever $\A\in \mathcal M_l(\l2)\cap \mathcal M_r(\l2)$ and we set
$$\|\A\|_{\mathcal M(\l2)}=\max \{\|\A\|_{\mathcal M_l(\l2)},\|\A\|_{\mathcal M_r(\l2)}\}.$$

 Denoting by $\A^*$ the adjoint matrix given by  $\Skj=T^*_{jk}$ for all $k,j\in \N$,  one easily sees that $\A\in \Bl2$ (respectively $\A\in \mathcal M_l(\l2)$) if and only if $\A^*\in \Bl2$ with $\|\A\|_{\Bl2}=\|\A^*\|_{\Bl2}$ (respectively  $\A^*\in \mathcal M_r(\l2)$ and $\|\A\|_{\mathcal M_l(\l2)}= \|\A^*\|_{\mathcal M_r(\l2)}$.)

It was shown in \cite[Theorem 4.7]{BB} that $\Bl2\subset \mathcal M_l(\l2)\cap \mathcal M_r(\l2)$.

We refer the reader to the books $\cite{DFS, DU}$ and to \cite{B} for the results in vector measures  and projective tensor products to be used in the sequel.

Recall that the space of operators $\B(L^1(\T),\B(H))$ can be identified with the space of measures $\mu\in V^\infty(\T,\B(H))$ using $\mu(E)=T(\chi_E)$ for any measurable set $E$. Hence, due to the well-known identification of  dual of projective tensor products, namely $(X\hat\otimes Y)^*=\B(X,Y^*)$, one clearly has $$\B(L^1(\T),\B(H))=V^\infty(\T,\B(H))=(L_1(\T)\hat\otimes (H\hat\otimes H)))^*.$$
In \cite[Theorem 5.1]{BB} it was shown that $\A=(\Tkj)\in \Bl2\cap \mathcal T$ if and only if there exists $\mu_\A\in V^\infty(\T,\B(H))$ such that $\widehat {\mu_\A}(j-k)=\Tkj$ for $j,k\in \N$.

Recall also that the space of regular vector valued measures of bounded variation $M(\T, \B(H))$ can be identified with $(C(\T, H\hat\otimes H))^*$ by Singer's theorem (see \cite{S}, \cite{S2}).  In \cite{BB} we use the notation
$M_{SOT}(\T, \B(H))$ for the space of regular measures $\mu$ such that $\mu_x(A)=\mu(A)(x)$ define a measure of bounded variation and
 $$\|\mu\|_{M_{SOT}(\T, \B(H))}= \sup_{\|x\|=1} |\mu_x|$$
 This space can be identified with $\B(H, M(\T, H))= (H\hat\otimes C(\T,H))^*$.
Of course $$\B(C(\T, H\hat\otimes H), \C)\subset \B(C(\T,H), H) \subset \B( C(\T), \B(H))$$
with  the natural norms and embedding between them.
 For each $k\in \Z$ we can define the Fourier coefficient  of  $T\in \B( C(\T), \B(H))$ as the operator $\hat T(k)\in \B(H)$ given by
$$\hat T(k)= T(\varphi_k), \quad \varphi_k(t)=e^{-ikt}.$$
Similarly if $\Phi\in \B( C(\T,H), H)$ we can associate the Fourier coefficient $\hat\Phi(k)\in \B(H)$ for $k\in \Z$ given by
$$\hat \Phi(k)(x)= \Phi(x\varphi_k), \quad x\in H.$$
For any regular $\B(H)$-valued measure $\mu$ we can associate an operator $T_\mu \in \B( C(\T), \B(H))$ given by
  \be \label{oper0}
T_\mu\left(\sum_{k=-N}^M \alpha_k \varphi_k\right)=\bigintsss \left(\sum_{k=-N}^M \alpha_k \varphi_k(t)\right)d\mu(t)
\ee and  we write $\hat \mu(l)=\widehat{T_\mu}(l)$ for $l\in \Z$. Moreover $\|T_\mu\|=\|\mu\|$ where $\|\mu\|$ stands for the semi-variation of the measure.

If $\mu$ is a regular vector measure such that  its adjoint measure $\mu^*$, given by $\mu^*(A)= \mu(A)^*$, belongs to $ M_{SOT}(\T,\B(H))$ we can also associate $\Phi_\mu \in \B(C(\T,H), H)$ given by
\be \label{oper2}
\Phi_\mu\left(\sum_{k=-N}^M x_k \varphi_k\right)=\sum_{k=-N}^M \widehat{\mu}(k)(x_k).
\ee
Since $(\Phi_\mu)^*: H\to M(\T,H)$ and $(\Phi_\mu)^*(x)=(\mu^*)_x$ (see \cite[Proposition 3.2]{BB})  we have
 that $\hat\mu (l)= \widehat{\Phi_\mu}(l)$  for $l\in \Z$ and $\|\Phi_\mu\|=\|\mu^*\|_{M_{SOT}(\T, \B(H))}$.

We would like to point out first that in \cite{BB} we showed that
$$M(\T, \B(H))\subset \M_r(\l2)\cap {\mathcal T} \subset M_{SOT}(\T, \B(H))$$
where with the inclusions we mean (see \cite[Theorem 5.4]{BB} ) that if $\A=(\hat \mu(j-k))$ for a given $\mu\in M(\T, \B(H))$ then $\A\in \M_r(\l2)$ with $\|\A\|_{\M_r(\l2)}\le \|\mu\|_{M(\T, \mathcal{B}(H))}$ and (see \cite[Theorem 5.7]{BB} ) that any Schur multiplier $\A=(\Tkj)\in \M_r(\l2)\cap {\mathcal T}$ satisfies that there exists $\mu\in M_{SOT}(\T, \mathcal{B}(H))$ such that $\Tkj=\hat\mu(j-k)$ and $\|\mu\|_{M_{SOT}(\T, B(H))}\le \|\A\|_{\M_r(\l2)}$.

Given $\eta\in M(\T)$ we shall denote by ${\bf M}_\eta$ the Toeplitz matrix given by  $${\bf M}_\eta= \left(\hat\eta(j-k) Id\right)_{k,j}\in \mathcal T.$$
  The cases $\eta=\delta_{-t}$  or $d\eta= f dt$ with $f\in L^1(\T)$ will be denoted by ${\bf M}_t$ and ${\bf M}_f$ respectively, that is ${\bf M}_t=(e^{i(j-k)t} Id)$ and  ${\bf M}_f=(\hat f(j-k) Id)$.
  In particular using (\ref{tt}) and (\ref{tb}), for $H=\C$, we have   that ${\bf M}_f\in \mathcal{B}(\ell^2)$ whenever $f\in L^\infty(\T)$ and that ${\bf M}_\eta\in \mathcal{M}(\ell^2)$ for any $\eta\in M(\T)$. One  basic observation is the relationship between Fourier multipliers and Schur multipliers coming from the formula
  \be \label{bf}
  {\bf M}_\eta* {\bf M}_f= {\bf M}_f*{\bf M}_\eta= {\bf M}_{\eta*f}
  \ee
 for any $ \eta\in M(\T)$ and $ f\in L^1(\T)$ where $\eta*f(t)=\int_0^{2\pi} f(e^{i(t-s)})d\eta(s)$ is the convolution between functions and measures in $\T$.

To make the connection between Fourier Analysis and Matriceal Analysis (see \cite{PP}) we shall introduce the following matrix-valued functions and operators.
\begin{defi} Let  $\A =(\Tkj)$ with $\Tkj \in \B(H)$ for $k,j\in \N$. Define
$$f_\A(t)= {\bf M}_t* \A = (e^{i(j-k)t}\Tkj), \quad t\in [0,2\pi].$$
\end{defi}
\begin{defi}
  Let $\A=(T_{j-k})\in \mathcal T$ and denote by $\varphi_l(t)=e^{-ilt}$ for $l\in \Z$. Define $T_\A:  P(\T)\to \B(H)$ and $\Phi_\A:  P(\T,H)\to H$ given by
$$T_\A\left(\sum_l \alpha_l \varphi_l\right)=  \sum_l \alpha_lT_l, \quad \sum_l \alpha_l \varphi_l\in  P(\T)$$
and
$$\Phi_\A\left(\sum_l x_l \varphi_l\right)=  \sum_l T_l(x_l), \quad \sum_l x_l \varphi_l\in P(\T,H).$$
\end{defi}
\begin{defi} Let $\A=(\Tkj)\in \mathcal U$. Define
$$F_\A(z)=  \sum_{l=0}^\infty \Dl z^l=(z^{(j-k)}\Tkj), \quad |z|<1,$$
and in the case $\A=(T_{j-k})\in \mathcal U\cap \mathcal T$, we write $$\tilde F_\A(z)=  \sum_{l=0}^\infty T_l z^l, \quad |z|<1.$$
\end{defi}

In particular for $z=re^{it}$ one has $$F_\A(z)= {\bf M}_{P_r}* {\bf M}_t* \A$$
where $P_r$ stands for the Poisson kernel. We shall use the notations
$$\sigma_n(\A)={\bf M}_{K_n}*\A, \quad P_r(\A)= {\bf M}_{P_r}*\A$$
where $K_n$ stands for the F\'ejer kernel.

In \cite{PP}, the space $\mathcal C(\ell^2)$ was introduced as those matrices in $\B(\ell^2)$ such that $\sigma_n(A)$ converges to $A$ in $\B(\ell^2)$.
We shall use a different approach  and  introduce such a class of matrices, to be called ``continuous matrices", with entries in the space $\B(H)$ and we shall see that it plays an important role in the study of Schur multipliers.
\begin{defi} Given a matrix $\A =(\Tkj)$ with entries $\Tkj\in \B(H)$ we say that $\A$ is a ``polynomial", in short  $\A\in\Pl2$, whenever there exists $N, M\in \N$ such that
$\A=\sum_{l=-N}^M\Dl$ and  \be \label{hip0} \sup_{j,k} \|\Tkj\|<\infty.\ee 
\end{defi}
Note that condition (\ref{hip0}) guarantees that $\Pl2\subset\Bl2$.
\begin{defi} 
We define $\Cl2$ as the closure of $\Pl2$ in $\Bl2$.
\end{defi}

The paper is divided into several sections. In Section 2 we deal with matrices in $\Bl2$. We observe that $\B(\ell^2)\hat\otimes \B(H)\subseteq \Bl2$ (see Example \ref{esc_to_op}) and  introduce  $\mathcal A(\l2)$, the analogue to the Wiener algebra, to produce easy examples in $\Cl2$. Section 3 contains results on Schur multipliers, for instance (see Proposition \ref{p2}) we show that $\M(\ell^2)\hat\otimes \B(H)\subseteq {\mathcal M}(\l2)$ and also  (see Proposition \ref{prop:sobre_las_sigmas}) that $\A\in \Bl2$ if and only if $\sup_n {\bf M}_{k_n}*\A <\infty$ where $k_n$ is a summability kernel.
The main results of the paper are in the last section, where we analyze the space $\Cl2$ in detail. First we relate (see Theorem \ref{thm:caract_cont}) that $\A\in \Cl2$ with the properties of $f_\A$ showing that it is equivalent to its continuity as $\Bl2$-valued function, giving also other characterizations such as the convergence
$P_r( \A)\to \A$  as $r\to 1$ or $\sigma_n(\A)\to \A$  as $n\to \infty$  in $\Bl2$.
We shall use this class of matrices to describe Schur multipliers, by showing (see Theorem \ref{thm:mult_c_to_c} ) that $\mathcal M_l(\ell^2(H))= (\Cl2, \Cl2)_l$. Two subsections are included in Section 4, one devoted to Toeplitz matrices and another one to upper triangular ones.
In particular we get (see Theorem \ref{prp:continuasoper}) the identification of ``continuous" Toeplitz matrices  $\A=(T_{j-k})$ as  functions $g\in C(\T, \B(H))$ with Fourier coefficients $\widehat{g}(l)=T_l$ for $l\in \Z$ and we also characterize (see Theorem \ref{thm1} and Theorem \ref{thm:caract_msot}) the left Toeplitz Schur multipliers acting on $\Bl2\cap \mathcal T$ as those matrices $\A$ such that $\Phi_\A$ extends to a bounded operator from $C(\T,H)$ into $H$. This completes and gives an alternative approach to the results in \cite{BB}.
Concerning upper triangular matrices, we show that if $\A\in \mathcal U$ then $\A\in \Bl2$ or $\A\in \Cl2$ if and only if $F_\A\in H^\infty(\D, \Bl2)$ or   $F_\A\in A(\D, \Bl2)$ respectively and similar results for $\A\in \mathcal U\cap \mathcal T$ in terms of $\tilde F_\A$ are presented.

\section {Matrices with entries in operators}
We shall present here some examples to have at our disposal when checking some analogues to classical results in our operator-valued setting.
As usual, given $x,y\in H$ we write $x\otimes y\in \B(H)$ for the rank one operator given by $(x\otimes y)(z)=\la z,x\ra y$ for $z\in H$.
The first trivial example is produced by tensoring elements in $\l2$.
\begin{ejem}
Let $\x=(x_j), \y=(y_j)\in \l2$.
Then  $$(\x\otimes \y)(\z)=\ll \z,\x\gg \y, \quad \z\in \l2$$
corresponds to the matrix $\A =(x_j\otimes y_k) $ and belongs to $\Bl2$.
Moreover  $\|\x\otimes \y\|_{\Bl2}=\|\x\|_{\l2}\|\y\|_{\l2}.$
\end{ejem}

\begin{ejem} Let  $\A =(\Tkj)$ and let $l\in \Z$ and $k,j\in \N$. Then

(i) $\Dl\in \Bl2$ iff $\sup_k \|T_{k,k+l}\|= \|\Dl\|_{\Bl2}<\infty.$

(ii) $\Cj\in \Bl2$ iff $\sup_{\|x\|=1} (\sum_{k=1}^\infty\|T_{kj}(x)\|^2)^{1/2}= \|\Cj\|_{\Bl2}<\infty.$

(iii) $\Rk\in \Bl2$ iff $\sup_{\|y\|=1} (\sum_{j=1}^\infty\|T^*_{kj}(y)\|^2)^{1/2}= \|\Rk\|_{\Bl2}<\infty.$
\end{ejem}
\begin{proof}
 (i) and (ii) are straightforward.

 (iii) follows by using  the duality $(\l2)^*=\l2$, since
\ba
\|\Rk\|&=&\sup_{\norm{\x}_{\l2}=1}\Norm{\sum_j T_{k,j}(x_j)}\\&=&\sup_{\substack{\norm{\x}_{\l2}=1\\\norm{y}=1}}\left|\sum_j\Presc{T_{k,j}(x_j)}{y}\right|\\
&=&\sup_{\substack{\norm{\x}_{\l2}=1\\\norm{y}=1}}\left|\sum_j\Presc{x_j}{T_{k,j}^*(y)}\right|\\
&=&\sup_{\norm{y}=1}\left(\sum_j\Norm{T_{k,j}^*(y)}^2\right)^{1/2}.
\ea
\end{proof}

\begin{ejem}\label{esc_to_op}
Let $A= (a_{k,j})\in \B(\ell^2)$ and $T\in \B(H)$. Then
$$\A=( a_{k,j}T)\in \Bl2 \quad \hbox{ and } \|\A\|_{\Bl2}= \|A\|_{\B(\ell^2)} \|T\|_{\B(H)}$$
\end{ejem}

\begin{proof}
 Let us take $\x,\y\in c_{00}(H)$, say $\x=(x_1,\cdots, x_N,0, \cdots)$ and $\y=(y_1,\cdots, y_M,0, \cdots)$. Let $(v_l)_{l=1}^\infty$ be an orthonormal basis  in $H$ and use the notation $x^l(j)$ for the coordinate with respect to such a basis, i.e. $x_j=\sum_{l=1}^\infty x^l(j)v_l$.

\begin{equation*}
\begin{split}
|\presc{\A \x}{\y}|&=\left|\sum_{k,j=1}^{N,M}a_{k,j}\presc{T(x_j)}{y_k}\right|=\left|\sum_{k,j=1}^{N,M}a_{k,j}\sum_l T(x_j)^ly^l(k)\right| \\
&=\left|\mathlarger{\mathlarger{\sum_l}}\left(\sum_{k,j=1}^{N,M}a_{k,j} T(x_j)^ly^l(k)\right)\right|\leq \mathlarger{\mathlarger{\sum_l}}\left|\sum_{k,j=1}^{N,M}a_{k,j} T(x_j)^ly^l(k)\right|\\
&\leq\mathlarger{\mathlarger{\sum_l}}\norm{A}_{\B(\ell^2)}\left(\sum_{j=1}^M|T(x_j)^l|^2\right)^\frac{1}{2}\left(\sum_{k=1}^N|y^l(k)|^2\right)^\frac{1}{2}\\
&\leq\norm{A}_{\B(\ell^2)}\left(\mathlarger{\mathlarger{\sum_l}}\left(\sum_{j=1}^M|T(x_j)^l|^2\right)\right)^\frac{1}{2}\left(\mathlarger{\mathlarger{\sum_l}}\left(\sum_{k=1}^N|y^l(k)|^2\right)\right)^\frac{1}{2}\\
&=\norm{A}_{\B(\ell^2)}\left(\sum_{j=1}^M\|T(x_j)\|^2\right)^\frac{1}{2}\left(\sum_{k=1}^N\|y_k\|^2\right)^\frac{1}{2}\\
&\leq\norm{A}_{\B(\ell^2)}\cdot\norm{T}\cdot\norm{\x}_{\ell^2(H)}\cdot\norm{\y}_{\ell^2(H)}.
\end{split}
\end{equation*}

 This gives us the first  inequality
$\normbh{\A}\leq \norm{T}\norm{A}_{\B(\ell^2)}.$

Now, selecting vectors of the form $\x=(\alpha_j x)$ where $x\in H$ with $\|x\|=1$  and $(\alpha_j)$ with $\norm{(\alpha_j)}_2=1$ we obtain the other inequality.
\begin{equation*}
\begin{split}
\normbh{\A}^2&\geq\sup_{\substack{\norm{(\alpha_j)}_2=1\\\norm{x}=1}}\mathlarger{\mathlarger{\sum_k}}\Norm{\sum_ja_{k,j}\alpha_jT(x)}^2 =\sup_{\substack{\norm{(\alpha_j)}_2=1\\\norm{x}=1}}\norm{Tx}^2\mathlarger{\mathlarger{\sum_k}}\left|\sum_ja_{k,j}\alpha_j\right|^2\\
&=\norm{T}^2\cdot\norm{A}^2_{\B(\ell^2)}.
\end{split}
\end{equation*}
\end{proof}

Let us define a class of matrices which generalize the Wiener algebra and which will give us examples in $\Cl2$.
\begin{defi} A matrix $\A=(\Tkj)$ is said to belong to $\mathcal A(\l2)$ whenever $$\|\A\|_{\mathcal A(\l2)}:=\sum_{l\in \Z} \sup_{k\in \N}\|T_{k,k+l}\|<\infty.$$
\end{defi}
\begin{nota} If $\A=(T_{j-k})\in \mathcal T \cap \mathcal A(\l2)$ then $(T_l)\in \ell^1(\Z,\B(H))$. Hence the function
$t\to \sum_{l\in \Z} T_l e^{ilt}$ belongs to the Wiener algebra $ A(\T, \B(H))$.
\end{nota}
\begin{nota} We refer the reader to \cite{Anca} for another possible extension of $A(\T)$  given in the case $H=\C$,  which also works for Toeplitz matrices, but which was not even contained in the space of bounded operators for general matrices.

\end{nota}
\begin{prop}
 $\mathcal A(\l2)\subsetneq \Cl2$.\end{prop}
 \begin{proof} Let $\A\in \mathcal A(\l2)$. Noticing that $\|\Dl\|_{\Bl2}=\sup_{k\in \N}\|T_{k,k+l}\|$, one gets
that $\sum_{l\in\Z} \Dl$ is absolutely convergent in $\Bl2$. Hence   $\A\in \Cl2$ and
$$\|\A\|_{\Bl2}\le \|\A\|_{\mathcal A(\l2)}.$$

It is immediate to see that ${\bf M}_f\in \mathcal A(\l2)$ if and only if $f\in A(\T)$.
Let us select $f\in C(\T)\setminus A(\T)$ and consider $\A= {\bf M}_f$.
Due to  Example \ref{esc_to_op} and (\ref{tt}) one has that  $$\|{\bf M}_{f-\sigma_n(f)}\|_{\Bl2}= \|f-\sigma_n(f)\|_{L^\infty(\T)}$$ where $\sigma_n(f)=K_n*f$ for the F\'ejer kernel $K_n$.
This   shows that ${\bf M}_f\in \Cl2 \setminus \mathcal A(\l2)$.
\end{proof}

\section{Schur multipliers for matrices with operator entries}

Recall (see \cite[Proposition 4.6]{BB}) that if $\A\in \Bl2$ then
$$\sup_{\|\x\|_{\l2}=1}\sum_{k,j} \|\Tkj(x_j)\|^2\le \|\A\|^2_{\Bl2}.$$
Let us give some necessary condition for Schur multipliers.

\begin{prop}
If $\A=(\Tkj)\in \mathcal M_r(\l2)$ then  $$\sup_{\|{\bf z}\|_{\l2}=1} \sup_k \left(\sum_{j=1}^\infty \|\Tkj(z_j)\|^2\right)^{1/2}\le \|\A\|_{\mathcal M_r(\l2)} .$$
\end{prop}
\begin{proof} Let $(\x\otimes \y)=(x_j\otimes y_k)\in \Bl2$  for some $\x,\y\in \l2$. Since $\A\in \mathcal M_r(\l2)$ then $(\x\otimes \y)*\A\in \Bl2$. Therefore
$$\|(\x\otimes \y)*\A(\z) \|_{\l2}\leq \|\A\|_{\mathcal M_r(\l2)} \|\x\|_{\l2}\|\y\|_{\l2}\|\z\|_{\l2}.$$ Using now that $(x_j\otimes y_k)\Tkj= \Tkj^*x_j\otimes y_k$ we obtain
$$\left\|\left((\sum_{j=1}^\infty \la \Tkj(z_j),x_j\ra) y_k\right)\right\|_{\l2}\le\|\A\|_{\mathcal M_r(\l2)} \|\x\|_{\l2}\|\y\|_{\l2}\|\z\|_{\l2}$$
which implies $$\sup_k\left|\sum_{j=1}^\infty \la \Tkj(z_j),x_j\ra\right|\le \|\A\|_{\mathcal M_r(\l2)} \|\x\|_{\l2}\|\z\|_{\l2}.$$ This, in particular, gives the result by duality.

\end{proof}

Let us show also the analogue to Example \ref{esc_to_op} for Schur multipliers.
\begin{prop} \label{p2} Let $A= (a_{kj})\in \M(\ell^2)$ and $T\in \B(H)$. Then
$\A=( a_{k,j}T)\in {\mathcal M}(\ell^2(H))$ and
$$ \|\A\|_{\mathcal M_l(\ell^2(H))}=\|\A\|_{\mathcal M_r(\ell^2(H))}= \|A\|_{\M(\ell^2)} \|T\|_{\B(H)}.$$
\end{prop}

\begin{proof} We deal only with  the case of left multipliers.
Let $\bold{B}=(B_{k,j})\in \mathcal{B}(\ell^2(H))$, and take $\x,\y\in \ell^2(H)$.

\begin{equation*}
\begin{split}
|\ll\A\ast \bold{B}(\x),\y\gg|&=\left|\sum_{k,j}a_{k,j}\Presc{T\circ B_{k,j}(x_j)}{y_k}\right|\\
&=\left|\mathlarger{\mathlarger{\sum_{k,j}}}a_{k,j}\Presc{T\circ B_{k,j}\left(\frac{x_j}{\norm{x_j}}\right)}{\frac{y_k}{\norm{y_k}}}\norm{x_j}\norm{y_k}\right|.
\end{split}
\end{equation*}

Since $A$ is a Schur multiplier, it verifies

$$|\sum_{k,j}a_{k,j}b_{k,j}\alpha_j\beta_k|\leq \norm{A}_{\M(\ell^2)}\norm{(\alpha_j)}_{\ell^2}\norm{(\beta_k)}_{\ell^2}\norm{B}_{\mathcal{B}(\ell^2)}, $$
$\forall B=(b_{k,j})\in \B(\ell^2),(\alpha_j),(\beta_k)\in\ell^2.$

Let us show that $\left(\Presc{T\circ B_{k,j}(\frac{x_j}{\norm{x_j}})}{\frac{y_k}{\norm{y_k}}}\right) \in \B(\ell^2)$. Indeed, using the boundedness of ${\bf B}$, we have for $\sum_j |\alpha_j|^2=\sum_k|\beta_k|^2=1$

\begin{eqnarray*}
\left|\mathlarger{\mathlarger{\sum}}_{k,j}{\Presc{T\circ B_{k,j}\left(\frac{x_j}{\norm{x_j}}\right)}{\frac{y_k}{\norm{y_k}}}}\alpha_j\beta_k\right|&=&\left|\mathlarger{\mathlarger{\sum}}_{k,j}{\Presc{T\circ B_{k,j}\left(\alpha_j\frac{x_j}{\norm{x_j}}\right)}{\beta_k\frac{y_k}{\norm{y_k}}}}\right|\\
&\le&\left(\mathlarger{\mathlarger{\sum}}_k\Norm{\sum_j T\circ B_{k,j}\left(\alpha_j\frac{x_j}{\norm{x_j}}\right)}^2\right)^\frac{1}{2}\\
&=&\left(\mathlarger{\mathlarger{\sum}}_k\Norm{T\sum_j  B_{k,j}\left(\alpha_j\frac{x_j}{\norm{x_j}}\right)}^2\right)^\frac{1}{2}\\
&\leq&\norm{T}\left(\mathlarger{\mathlarger{\sum}}_k\Norm{\sum_j  B_{k,j}\left(\alpha_j\frac{x_j}{\norm{x_j}}\right)}^2\right)^\frac{1}{2}\\
&\le& \norm{T}\norm{{\bf {B}}}_{\mathcal{B}(\ell^2(H))}.
\end{eqnarray*}

Therefore,

$$|\ll\A\ast \bold{B}(\x),\y\gg| \leq \norm{\A}_{\mathcal{M}(\ell^2)}\norm{T}\norm{ {\bf B}}_{\Bl2}\norm{\x}_{\ell^2(H)}\norm{\y}_{\ell^2(H)},$$

and taking supremums, we get the desired inequality

$$\norm{\A}_{\mathcal M_l(\ell^2(H))}\leq\norm{T}\cdot \norm{A}_{\mathcal{M}(\ell^2)}.$$

To check that $\norm{\A}_{\mathcal M_l(\ell^2(H))}\geq\norm{T}\cdot \norm{A}_{\mathcal{M}(\ell^2)}$, we select a particular set of matrices from $\B(\ell^2(H))$ and apply Example \ref{esc_to_op}.

\begin{equation*}
\begin{split}
\norm{\A}_{\mathcal M_l(\ell^2(H))}&=\sup_{\norm{\bold{B}}_{\mathcal{B}(\ell^2(H))}=1}\norm{(a_{k,j} T)\ast {\bf B}}_{\mathcal{B}(\ell^2(H))} \\
&\geq \sup_{\norm{(b_{k,j}Id)}_{\mathcal{B}(\ell^2(H))}=1}\norm{(a_{k,j} T)\ast (b_{k,j}Id)}_{\mathcal{B}(\ell^2(H))}\\
&=\sup_{\normbh{(b_{k,j} Id)}=1}\norm{(a_{k,j}b_{k,j}T)}_{\mathcal{B}(\ell^2(H))}\\
&=\sup_{\norm{(b_{k,j})}_{\mathcal{B}(\ell^2)}=1}\norm{T}\cdot \norm{(a_{k,j}b_{k,j})}_{\mathcal{B}(\ell^2)}\\
&=\norm{T}\cdot\norm{A}_{\mathcal{M}(\ell^2)}
\end{split}
\end{equation*}

The case $\A\in \M_r(\l2)$ follows the same lines and it is left to the reader.
\end{proof}

Let us give a characterization of matrices in $\Bl2$ using  Schur products.
First, we recall the definition of a summability kernel.

A sequence $\lbrace{k_n\rbrace}\subset L^1(\mathbb{T})$ is called a ``summability kernel'' (also denoted a ``bounded approximation of the identity'') if it satisfies the following properties:

1) $\frac{1}{2\pi}\int_{-\pi}^\pi k_n(t) dt=1$.
\medskip

2) $\sup_{n\in \N}\frac{1}{2\pi}\int_{-\pi}^\pi |k_n(t)| dt=C<\infty$.
\medskip

3) $\forall 0<\delta<\pi$ one has $\frac{1}{2\pi}\int_{\delta\leq|t|\leq\pi} k_n(t) dt\xrightarrow[n\to\infty]{}0$.
\bigskip

Classical examples to be used in the sequel are
the F\'ejer kernel $$K_n(t)=\mathlarger{\mathlarger{\sum}}_{k=-n}^n\left( 1- \frac{|k|}{n+1}\right)e^{ikt}$$ and the Poisson kernel $$ P_r(t)=\sum_{k\in \Z} r^{|k|}e^{ikt}.$$ As mentioned in the introduction,
$\sigma_n(\A)= {\bf M}_{K_n}* \A$ and
$ P_r(\A)=  {\bf M}_{P_r}* \A$ for $\A=(T_{k,j})$.

Observe that under the assumption (\ref{hip0}) one has $\sigma_n(\A)\in {\mathcal P}(\l2)$ and $ P_r(\A)\in \mathcal{A}(\l2)$, since $\sup_l \|{\bf D}_l\|<\infty$.
The reader should be aware that condition (\ref{hip0}) holds for any $\A\in \Bl2$ since for $j,k\in \N$, $$\|\Tkj\|= \sup_{\|x\|=\|y\|=1} \ll \A x\ej, y\ek\gg. $$
\begin{lema}\label{lemma:sigma_n_wot}
Let $\A\in\Bl2$ and $\lbrace{k_n\rbrace}$ a summability kernel, and denote $M_n(\A)= {\bf M}_{k_n}* \A$. Then, $M_n(\A)\xrightarrow[n\to\infty]{WOT} \A$. In other words,
\be \label{l0}\ll M_n(\A)\x, {\y}\gg\xrightarrow[n\to\infty]{}\ll\A \x, {\y}\gg\quad\text{for all } \x,\y\in \l2.\ee
\end{lema}
\begin{proof}
 Formally we have $\A-M_n(\A)= \sum_{l\in \Z}(1-\hat k_n(l)){\bf D}_l$ and for each $x,y\in H$ and $j,k\in \N$ we can write
$$\ll(\A-M_n(\A))x{\bf e}_j, y{\bf e}_k\gg= (1-\hat k_n(j-k)) \la \Tkj x, y\ra.
   $$
   Now, since $\lbrace{ k_n\rbrace}$ is a summability kernel, we know that $\lim_n k_n \ast g=g$ $\forall  g\in L^1(\mathbb{T})$, therefore  $\lim_n \hat{k}_n(l)=1$ $\forall l\in\mathbb{Z}$, and we can conclude that $\ll(\A-M_n(\A))x{\bf e}_j, y{\bf e}_k\gg \xrightarrow[n\to\infty]{} 0$.
   In particular, (\ref{l0}) holds for $\x,\y\in c_{00}(H)$. Now let
$\x\in \ell^2(H)$ and $\y\in \ell^2(H)$ and take sequences $(\x_N)_N\subset c_{00}(H)$ and $(\y_N)_N\subset c_{00}(H)$ such that $\norm{\x-\x_N}_{\l2}\to 0$ and $\norm{\y-\y_N}_{\l2}\to 0$. Then,
\begin{eqnarray*}
|\ll(\A- M_n(\A))\x,\y\gg|&=&
|\ll(\A-M_n(\A))(\x-\x_N),\y\gg|\\
&+& |\ll(\A-M_n(\A))\x_N,\y-\y_N\gg|\\
&+&|\ll(\A-M_n(\A))\x_N,\y_N\gg|\\
&\leq &\norm{\A-M_n(\A)}_{\Bl2}\norm{\x-\x_N}_{\l2}\norm{\y}_{\l2}\\
&+&\norm{\A-M_n(\A)}_{\Bl2}\norm{\x_N}_{\l2}\norm{\y-\y_N}_{\l2}\\
&+& |\ll(\A-M_n(\A))\x_N,\y_N\gg|.
\end{eqnarray*}
Combining now Proposition \ref{p2} and (\ref{tb}) we obtain\be \label{estima1} \|M_n(\A)\|_{\Bl2}\le \|{\bf M}_n\|_{\Ml2}\|\A\|_{\Bl2}=\|k_n\|_{L^1(\T)}\|\A\|_{\Bl2}\ee
  and, therefore, using the estimates   $\norm{\A-M_n(\A)}_{\Bl2}\leq (1+C)\norm{\A}_{\Bl2}$  and $\|\x_N\|_{\l2}\le \|\x\|_{\l2}$  and letting $N\to \infty$ we finish the proof.
\end{proof}

\begin{prop}\label{prop:sobre_las_sigmas}
 Let $\A$ be a matrix with entries in $\mathcal{B}(H)$ and $\lbrace{k_n\rbrace}$ a summability kernel, and denote $M_n(\A)= {\bf M}_{k_n}* \A$. Then:

(i) $ \A\in \B(\ell^2(H))\;\Leftrightarrow\; \sup_n\normbh{M_n(\A)}<\infty$

(ii) $\A\in \Ml2\;\Leftrightarrow\; \sup_n\normmh{M_n(\A)}<\infty$

(iii) $\A\in {\mathcal M}_l(\l2)\;\Leftrightarrow\; \sup_n\norm{M_n(\A)}_{{\mathcal M}_l(\l2)}<\infty$
  \end{prop}

  \begin{proof}
  (i) If $\A\in \Bl2$ then (\ref{estima1}) gives  $\sup_n\normbh{M_n(\A)}<\infty$. Assume now that $\sup_n\normbh{M_n(\A)}<\infty$. As a consequence of lemma \ref{lemma:sigma_n_wot}, one gets that
  $$\ll \A\x, \y\gg =\lim_n \ll M_n(\A)\x, \y\gg.$$
  Hence $|\ll \A\x, \y\gg |\le C\|\x\|_{\l2}\|\y\|_{\l2}$ and therefore $\A\in \Bl2)$.

(ii) Assume that $\A\in \Ml2$. Since $M_n(\A)={\bf M}_{k_n}* \A=\A* {\bf M}_{k_n}$ we have:
$$
\normmh{ {\bf M}_{k_n}\ast \A}
\leq \normmh{\A}\cdot\normmh{{\bf M}_{k_n}}=\normmh{\A}\|k_n\|_{L^1(\T)}.
$$
And taking the supremum over $n$ we conclude that
$$\sup_n \normmh{M_n(\A)}\leq C \normmh{\A}< \infty.$$

Assume $\sup_n\normmh{M_n(\A)}= K<\infty$, and take ${\bf B}\in \Bl2$. We have
$$\normbh{M_n({\bf B}\ast \A)} =\normbh{{\bf B}\ast M_n(\A))}\leq K\normbh{{\bf B}}.$$

   And using again that $M_n({\bf B}\ast \A)$ converges in the weak operator topology to ${\bf B}\ast \A$ (as shown in Lemma \ref{lemma:sigma_n_wot}), we have that $ {\bf B}\ast \A\in \Bl2$. Therefore, $\A\in \Ml2$.

   (iii) follows the same ideas.
  \end{proof}

\section{The space of ``continuous" matrices}

Recall from the introduction that $\Cl2$ stands for those matrices that can be approximated in the operator norm by matrices with a finite number of diagonals, called polynomials in ${\mathcal P}(\l2)$. We shall try to motivate the name of ``continuous " matrices, using the following function associated to a matrix given in the introduction.
If  $\A=(\Tkj)$ we defined $$f_\A(t)={\bf M}_t*\A=(e^{i(j-k)t}\Tkj), \quad t\in [-\pi,\pi).$$
Clearly if  $\A\in \Pl2$ one has
$$f_\A(t)=\sum_{l\in \Z} \Dl e^{ilt}\in  P(\T, \Bl2).$$

Let us see first that such a function takes values in spheres of $\Bl2$ or $\Ml2$ whenever $\A$ belongs to $\Bl2$ or $\Ml2$.
\begin{prop}\label{prop_norma_f_a} Let $\A=(\Tkj)$ be a matrix and $t\in [-\pi,\pi)$.

(i) If $\A\in \Bl2$ then $\norm{f_\A(t)}_{\Bl2}=\norm{\A}_{\Bl2}$.

(ii) If $\A\in \Ml2$ then $\norm{f_\A(t)}_{\Ml2}=\norm{\A}_{\Ml2}$.

(iii) If $\A\in {\mathcal M}_l(\l2)$ then $\norm{f_\A(t)}_{{\mathcal M}_l(\l2)}=\norm{\A}_{{\mathcal M}_l(\l2)}$.
\end{prop}
\begin{proof}
(i) and (ii) follow using that $\|{\bf M}_t\|_{\Ml2}= \|\delta_{-t}\|_{M(\T)}=1$ and the fact that $\A= f_\A(t)\ast {\bf M}_{-t}$.

(iii) follows from (ii) taking adjoints or repeating the argument above.
\end{proof}

\begin{prop}  Let $\A=(\Tkj)$ where $\Tkj=0$ for each $j\ne 2k$ and $T_{k,2k}=I$ for $k\in \N$. Then $\A\in \Bl2$ and  $f_\A$ is not strongly measurable with values in $\Bl2$.
\end{prop}
\begin{proof}
 It is clear that $\A\in \Bl2$, since $\A\x= (x_{2k})_k$. However, taking $\x=(x_k)\in \ell^2(H)$ , we observe that
$$(f_\A(t)-f_\A(s))(\x)=\left((e^{ikt}-e^{iks})\cdot x_{2k}\right)_k.$$

Hence if $t\ne s$,
\begin{eqnarray*}
\normbh{f_\A(t)-f_\A(s)}&=&\sup_{\|\x\|_{\l2}\leq 1}\sqrt{\sum_k |e^{i(t-s)k}-1|^2\norm{x_{2k-1}}^2}\\
&=&\sup_{k\in\mathbb{N}}|e^{ik(t-s)}-1|\geq \sqrt 2.
\end{eqnarray*}
Therefore $\{f_\A(t): t\in [-\pi,\pi]\}$ is not separable in $\Bl2$. Hence $t\to f_\A(t)$ is not strongly measurable by Pettis' measurability theorem (see \cite{DU}).
\end{proof}

\begin{prop}  If $\A\in \Bl2$ then $f_\A$ is continuous in the strong operator topology, that is for each $\x\in \l2$, the map
$$t\to f_\A(t)(\x)=\left(e^{-ikt}\sum_{j=1}^\infty \Tkj(x_j)e^{ijt}\right)_k$$
belongs to $C(\T,\l2)$.
\end{prop}

\begin{proof}
Assume first that $\x=x{\bf e}_j\in \ell^2(H)$ for some $x\in H$ and $j\in \N$.
Hence
$$f_\A(t)(x{\bf e}_j)=\left(T_{k,j}(x)e^{i(j-k)t}\right)_k.$$
Since $\A\in \Bl2$ we have that the series $\sum_k\norm{T_{k,j}(x)}^2$ is convergent. Then,  given  $\varepsilon>0$, select $N\in\N$ such that $\sum_{k=N}^\infty\norm{T_{k,j}(x)}^2<\varepsilon/4$.  Let $\varepsilon_j=\frac{\varepsilon}{2\cdot(N-1)\cdot\displaystyle\sup_{k<N}{\norm{T_{k,j}(x)}}}$ and let $\delta_j$ be defined such that $$ \displaystyle\sup_{k<N}|e^{i(j-k)u}-1|^2<\varepsilon_j, \quad \hbox{ for all } |u|<\delta_j.$$

Hence, if $s,t\in [-\pi,\pi)$ are such that $|s-t|<\delta_j$ one has
\begin{eqnarray*}
\|f_\A(t)(x{\bf e}_j)-f_\A(s)(x{\bf e}_j)\|_{\l2}^2&=&\sum_k\Norm{T_{k,j}(x)(e^{i(j-k)t}-e^{i(j-k)s})}^2\\
&=&\sum_{k=1}^{N-1}\Norm{T_{k,j}(x)(e^{i(j-k)t}-e^{i(j-k)s})}^2\\
&+&\sum_{k=N}^\infty\Norm{T_{k,j}(x)(e^{i(j-k)t}-e^{i(j-k)s})}^2\\
&\leq &\sup_{k<N}{\Norm{T_{k,j}(x)}}\cdot\sum_{k=1}^{N-1}\left|e^{i(j-k)t}-e^{i(j-k)s}\right|^2\\
&+& 2\cdot \varepsilon/4  \\
& \leq&\sup_{k<N}{\Norm{T_{k,j}(x)}}\cdot\sum_{k=1}^{N-1}\left|e^{i(j-k)(t-s)}-1\right|^2\\
&+&\varepsilon/2\;<\varepsilon.
\end{eqnarray*}

So we proved that $f_\A(t)(x{\bf e}_j)$ is continuous. Hence $f_\A(t)(\x)$ is also continuous for $\x\in c_{00}(H)$.
To prove the general case, consider $\x\in \ell^2(H)$, and select a sequence $(\x_N)_N:=((x_i)_{i=1}^{N})_N\subset c_{00}(H)$ converging to \x. Then
$$f_\A(t)(\x)= f_\A(t)(\x-\x_N)+f_\A(t)(\x_N)$$
and  invoke (i) in Proposition \ref{prop_norma_f_a} to get \begin{eqnarray*}\sup_t \norm{f_\A(t)(\x-\x_N)}_{\l2}&\leq& \sup_t \norm{f_\A(t)}_{\Bl2}\norm{\x-\x_N}_{\l2}\\
&=& \norm{\A}_{\Bl2}\cdot\norm{\x-\x_N}_{\l2}\xrightarrow[N\to\infty]{}0.\end{eqnarray*}
Therefore, $t\to f_\A(t)(\x)$  is a uniform limit of the continuous functions $t\to f_\A(t)(\x_N)$ and hence continuous.
\end{proof}

Let us now give several characterizations of matrices in $\Cl2$.

\begin{teor} \label{thm:caract_cont}
Let $\A$ be a matrix whose entries are in $\mathcal{B}(H)$. The following are equivalent:

1) $\A\in \Cl2$.

2) $\lim_{n\to \infty} M_n(\A)=\A$ in $\Bl2$
where $M_n(\A)=\bold{M}_{k_n}\ast \A$ and $\lbrace{k_n\rbrace}\subseteq L^1(\T)$ is a summability kernel.

3) $\lim_{n\to \infty} \sigma_n(\A)=\A$ in $\Bl2$.

4) $t\to f_\A(t)$ is a $\Bl2$-valued continuous function.

\end{teor}

\begin{proof}

1)$\Rightarrow$ 2). Let $\varepsilon>0$, and select $\bold{P}=(S_{k,j})=\sum_{l=-N}^{N}{\bf D_l}\in\Pl2$ such that $\normbh{\A-\bold{P}}<\varepsilon/3C$ where
$C=\sup_n \|k_n\|_{L^1(\T)}\ge 1$ Then,
\ba
\norm{M_n(\bold{P})-\bold{P}}_{\Bl2}&=&\Norm{\sum_{l=-N}^{N}(\hat{k_n}(l)-1){\bf D_l}}_{\Bl2}\\
&\leq&\sup_{k,j}\norm{S_{k,j}}\cdot (2N+1)\cdot \max_{|l|\leq N}|\hat{k_n}(l)-1|\ea
 Since $\lbrace{k_n\rbrace}$ is a summability kernel, one has that $\hat k_n(l)\to 1$ as $n\to\infty$ $\forall l\in \Z$. So, we can choose $n_0\in\mathbb{N}$ such that $|\hat{k}_n(l)-1|<\frac{\varepsilon}{3(2N+1) \sup_{k,j}\norm{S_{k,j}}}$ $\forall n\geq n_0$ and $\forall |l|\leq N$. Hence, $\norm{M_n(\bold{P})-\bold{P}}_{\Bl2}<\varepsilon/3$.   Finally, for $n\ge n_0$,
 \begin{eqnarray*}
\normbh{M_n( \A)-\A}
&\leq&\normbh{\bold{M}_{k_n}\ast (\A-{\bf P})}\\
&+&\normbh{\bold{M}_{k_n}\ast \bold{P}-\bold{P}}+\normbh{\bold{P}-\A} \\
&\leq&\normmh{\bold{M}_{k_n}}\cdot\normbh{\A-\bold{P}}+\varepsilon/3+\varepsilon/3 \\
&\leq&\norm{k_n}_{L^1}\cdot \varepsilon/3C+\varepsilon/3+\varepsilon/3=\varepsilon. \\
\end{eqnarray*}

2)$\Rightarrow$ 3). It is obvious since the Fej\'er Kernel is an example of summability kernel.

 3)$\Rightarrow$ 4).  Observe that $\sigma_n(f_\A(t))=f_{\sigma_n(\A)}(t) \in P(\T, \Bl2)$.
Since
\ba
\sup_t\normbh{f_{\sigma_n(\A)}(t)-f_\A(t)}&=&\sup_t\normbh{f_{\sigma_n(\A)-\A}(t)}\\
&=&\normbh{\sigma_n(\A)-\A}\ea
then $f_\A$ is a uniform limit of $\Bl2$-valued  polynomials, hence it is a continuous function.
\bigskip

 4) $\Rightarrow$ 1). Since $f_\A\in C(\T,\Bl2)$ then it can be approximated in norm by polynomials in $P(\T, \Bl2)$. Using again
$$\normbh{{\bf P}-\A}{=}\sup_t\normbh{f_{\bf P}(t)-f_\A(t)}$$
one obtains the result.
\end{proof}

\begin{coro} $\mathcal A(\l2)$ is dense in $\Cl2$.
\end{coro}
\begin{proof} Given $\A\in \Cl2$ one has that $P_r(\A)={\bf{ M}}_{P_r}*\A$ converges to $\A$ in $\Bl2$. Since
$P_r(\A)=\sum_l r^{|l|}\Dl \in \mathcal A(\l2)$ for each $0<r<1$  the result follows.
\end{proof}

\begin{nota} $\Cl2$ is an ideal of $\Bl2$, that is to say if $\A\in \Cl2$ and ${\bf B}\in \Bl2$ then
$ \A\ast {\bf B}\in \Cl2 \hbox{ and } {\bf B}\ast \A\in \Cl2.$
\end{nota}

\begin{defi} We write $(\Cl2,\Cl2)_l$  for the set of matrices $\A$ such that
$$ \A\ast {\bf B}\in \Cl2 \quad \forall {\bf B}\in \Cl2. $$
 And similar definition for  $(\Cl2,\Cl2)_r$.
\end{defi}

\begin{teor} \label{thm:mult_c_to_c}
$ \A\in \M_l(\l2)$ (respectively $\A\in \M_r(\l2)$) if and only if $\A\in (\Cl2,\Cl2)_l$
(respectively $\A\in (\Cl2,\Cl2)_r$).
\end{teor}

\begin{proof}
We shall give the proof for left Schur multipliers.
Let us assume that $\A\in\M_l(\l2)$ and  $\bold{B}\in \Cl2$. To show that  $\A\ast \bold{B}\in \Cl2$ observe that $\sigma_n(\A\ast \bold{B})=\A\ast \sigma_n({\bf B})$ and take limits as $n\to \infty$.

Suppose now that $\A\in (\Cl2,\Cl2)_l$, and let $\bold{B}\in\Bl2$. Note that $\sigma_n(\bold{B})\in \Pl2\subset \mathcal{C}(\ell^2(H))$. Therefore, by hypothesis, we have that for all $n\in\mathbb{N}$,
\begin{eqnarray*}
\normbh{\sigma_n(\A\ast \bold{B})}&=&\normbh{\A\ast\sigma_n(\bold{B})}\\
&\leq&\norm{\A}_{(\Cl2,\Cl2)_l}\cdot\normbh{\sigma_n(\bold{B})}\\
&\leq&\norm{\A}_{(\Cl2,\Cl2)_l}\cdot\normbh{\bold{B}}.
\end{eqnarray*}

This, in particular, means that for all $\x,\y\in \l2$ with norm $1$,
$$|\ll\sigma_n(\A\ast \bold{B})(\x),\y\gg |\leq\norm{\A}_{(\Cl2,\Cl2)_l}\cdot\normbh{\bold{B}}\quad \forall n\in\mathbb{N}.$$
A use of Lemma \ref{lemma:sigma_n_wot} gives us that
$$|\ll \A\ast \bold{B}(\x),\y\gg |\leq\norm{\A}_{(\Cl2,\Cl2)_l}\cdot\normbh{\bold{B}},$$
and therefore
$$\normbh{\A\ast \bold{B}}\leq\norm{\A}_{(\Cl2,\Cl2)_l}\cdot\normbh{\bold{B}},$$
so $\A\in \M_l(\l2)$ and $\|\A\|_{\mathcal{M}_l(\ell^2(H))}\le \|\A\|_{(\Cl2,\Cl2)_l}.$
\end{proof}

\subsection{The Toeplitz case}

In this section we restrict ourselves to the case of Toeplitz matrices.
We write $\Bl2_{\mathcal T}= \Bl2 \cap {\mathcal T}$ and $\Cl2_{\mathcal T}= \Cl2 \cap {\mathcal T}$.

Although $\Bl2_\mathcal T$ can be identified (as mentioned in the introduction) with $V^\infty(\T,\B(H))$, we shall give a direct proof of the corresponding result in the setting of continuous matrices, not relying on that fact. We shall use next lemma.
\begin{lema}
\label{lema:norma}
Let $f\in C(\mathbb{T},\mathcal{B}(H))$, and consider $\A_f=(T_{k,j})_{k,j}$ with $T_{k,j}:=\hat{f}(j-k)$. Then $\A_f\in \Cl2_{\mathcal T}$ with
$$\normbh{\A_f}=\norm{f}_{C(\T,\Bl2)}$$

\end{lema}

\begin{proof}
Let $\x=(x_j),\y=(y_k)\in \ell^2(H)$ and observe  that
\ba   & &\bigintss_{0}^{2\pi} \left\lVert \sum_j   x_{j}e^{-ijt}\otimes  \sum_ky_k e^{ikt}\right\rVert_{H\hat{\otimes}H}\; \frac{dt}{2\pi}=\bigintss_0^{2\pi} \Norm{\sum_j  x_{j}e^{-ijt}}\Norm{\sum_ky_k e^{ikt}}\frac{dt}{2\pi}\\
  &\le &   \left(\bigintss_0^{2\pi} \Norm{\sum_j  x_{j}e^{-ijt}}^2\frac{dt}{2\pi}\right)^{1/2}\left(\bigintss_0^{2\pi}\Norm{\sum_ky_k e^{ikt}}^2\frac{dt}{2\pi}\right)^{1/2}\\
  &=&\|\x\|_{\ell^2(H)}\|\y\|_{\ell^2(H)}.
  \ea
  Hence $(\sum_j  x_{j}e^{-ijt})\otimes (\sum_ky_k e^{ikt})\in L^1(\T, H\hat\otimes H)$. Identifying $\B(H)=(H\hat\otimes H)^*$ by means of the formula
  $T(x\otimes y)=\la Tx,y\ra$, we can write

\begin{eqnarray*}
|\ll \A_f\x,\y\gg|&=&\left|\sum_{k,j}\presc{T_{k,j}x_j}{y_k}\right|=\left|\sum_{k,l}\presc{T_lx_{l+k}}{y_k}\right|\\
&=&\left|\sum_l\ T_l\Big(\sum_k x_{l+k}\otimes y_k\Big)\right|\\
&=&\left|\bigintss_0^{2\pi}\left(\sum_l\hat f(l) e^{ilt}\right)\left(\mathlarger{\mathlarger{\sum}}_l \left(\sum_k x_{l+k}\otimes y_k\right) e^{-ilt}\right)\frac{dt}{2\pi}\right|\\
&=&\left|\bigintss_0^{2\pi} f(t)\left(\left(\sum_j  x_{j}e^{-ijt}\right)\otimes \left(\sum_ky_k e^{ikt}\right)\right)\frac{dt}{2\pi}\right|\\
&\leq&\bigintss_0^{2\pi} \norm{f(t)}_{\mathcal{B}(\ell^2(H))}\Norm{\left(\sum_j  x_{j}e^{-ijt}\right)\otimes \left(\sum_ky_k e^{ikt}\right)}_{H\hat\otimes H}\frac{dt}{2\pi}\\
&=& \norm{f}_{C(\T,\Bl2)} \bigintss_0^{2\pi} \Norm{\sum_j  x_{j}e^{-ijt}}\cdot \Norm{\sum_ky_k e^{ikt}}\frac{dt}{2\pi}\\
&=& \norm{f}_{C(\T,\Bl2)}\norm{\x}_{\ell^2(H)}\norm{\y}_{\ell^2(H)}.
\end{eqnarray*}

Therefore $\normbh{\A_f}\leq\norm{f}_{C(\T,\Bl2)}$.

Let us check the other inequality. In order to do that, select $(x\alpha_j)$ and $(y\beta_k)$, where $x,y\in H$ are unitary and $(\alpha_j),(\beta_k)$ belong to the unit sphere of $\ell^2$. Hence
\begin{equation*}
\begin{split}
\left|\Presc{\A_f(x\alpha_j)}{y \beta_k}\right|&=\left|\sum_{k,j}\presc{T_{j-k}x}{y}\alpha_j\beta_k\right|=\left|\sum_{l}\presc{T_lx}{y}\sum_k\alpha_{l+k}\beta_k\right|\\
&=\left|\bigintsss_0^{2\pi} \left(\sum_l\presc{T_{l}x}{y}e^{ilt}\right)\left(\sum_l \sum_k \alpha_{l+k}\beta_k e^{-ilt}\right)\frac{dt}{2\pi}\right|\\
&=\left|\bigintsss_0^{2\pi} \left(\sum_l\presc{T_{l}x}{y}e^{ilt}\right)\left(\sum_j\alpha_je^{ijt} \sum_k \beta_{k} e^{-ikt}\right)\frac{dt}{2\pi}\right|\\
\end{split}
\end{equation*}

And taking into account that
$$\sup_t\Norm{\sum_l\presc{T_{l}x}{y}e^{ilt}}= \sup_{\|g\|_{L^1(\T)}=1} \left|\bigintsss_0^{2\pi} \sum_l\presc{T_{l}x}{y}e^{ilt} g(t) \frac{dt}{2\pi}\right|$$
together with the factorization $L^2(\T)\cdot L^2(\T)= L^1(\T)$ we obtain that
$$\sup_t\Norm{\sum_l\presc{T_{l}x}{y}e^{ilt}}=\sup_{\substack{\norm{(\alpha_j)}=1\\\norm{(\beta_k)}=1}}|\presc{\A_f(x\alpha_j)}{y \beta_k}|\leq\norm{\A_f}_{\Bl2}\|x\|\|y\|.$$
Finally, since
\ba
\norm{f}_{C(\T,\B(H))}&=&\sup_t\sup_{\substack{\norm{x}=1\\\norm{y}=1}}
\Norm{\sum_l\presc{T_{l}x}{y}e^{ilt}}\\
&=&\sup_{\substack{\norm{x}=1\\\norm{y}=1}}\sup_t\Norm{\sum_l\presc{T_{l}x}{y}e^{ilt}}\\
&\leq&\norm{\A_f}_{\Bl2}\ea
we conclude that $\A_f\in \Bl2_{\mathcal T}$ and $\norm{f}_{C(\T,\B(H))}=\norm{\A_f}_{\Bl2}$. To obtain that $\A_f\in \Cl2$ simply observe that
if $P$ is a polynomial in ${\mathcal P}(\T, \B(H))$ then $\A_P\in {\mathcal P}(\l2)$ and the proof is complete using an approximation argument.
\end{proof}

When we consider the Toeplitz case, we have the following result that identifies $\Cl2$ with the space of continuous functions.

\begin{teor}
\label{prp:continuasoper}
Let $(T_l)_{l\in \Z}$ be a sequence of operators in $\B(H)$ and let $\A=(T_{j-k})_{k,j}$. Then,
$\A\in \mathcal{C}(\ell^2(H))_{\mathcal T}$ if and only if there exists
$$  g_\A\in C(\mathbb{T},\mathcal{B}(H))\text{ such that } \widehat{g_\A}(l)=T_l.$$
Furthermore, $\norm{g_\A}_{C(\T,\B(H))}=\norm{\A}_{\Bl2}.$
\end{teor}

\begin{proof}
Assume first that $\A\in \Pl2$. Then choose $g_\A=\sum_{l=-N}^N T_l e^{ilt}\in P(\T,\B(H))$ and, due to Lemma  \ref{lema:norma}, $\normbh{\A}=\normbh{\A_{g_\A}}=\norm{g_\A}_{C(\T,\B(H))}$.

Now for the general case, if $\A\in \Cl2$ we use that $\sigma_n(\A)\in\Pl2$ converges to  $\A$ in $\Bl2$.  In particular $(\sigma_n(\A))_n$ is a Cauchy sequence in $\Bl2$. Using now Lemma  \ref{lema:norma} we get that for any $n,m$, $\normbh{\sigma_n(\A)-\sigma_m(\A)}=\norm{g_{\sigma_n(\A)}-g_{\sigma_m(\A)}}_{C(\T,\B(H))}$. Hence, the sequence $(g_n)_n:=(g_{\sigma_n(\A)})_n$  is a Cauchy sequence in $C(\T,\B(H))$, so it has a limit $g_\A\in C(\T,\mathcal{B}(H))$. Clearly  $\widehat{g_\A}(l)=T_l$ for each $l\in \Z$ due to the fact  $\widehat{g_n}(l)\to \widehat{g_\A}(l)$.
Also
$$\normbh{\A}=\lim_n\normbh{\sigma_n(\A)}=\lim_n\norm{g_n}_{C(\T,\B(H)}=\norm{g_\A}_{C(\T,\B(H))}.$$

Conversely, let us assume that there exists $g_\A\in C(\mathbb{T},\mathcal{B}(H))$\ such that $\hat{g_\A}(l)=T_l$.  Then, applying Lemma \ref{lema:norma}  we obtain that $\A_{g_\A}\in \mathcal{C}(\ell^2(H))_{\mathcal T}$ and $\norm{g_\A}_{C(\T,\B(H))}=\norm{\A_{g_\A}}_{\Bl2}$. Since $\Tkj= T_{j-k}=\widehat{g_\A}(j-k)$ we have that $\A_{g_\A}=\A$ and the converse implication is shown.
\end{proof}

We present now another proof of Theorem 5.7 in \cite{BB} making use of Theorems \ref{thm:mult_c_to_c} and \ref{prp:continuasoper}.
\begin{teor} \label{thm1}
Let $(T_l)_{l\in \Z}$ be a sequence of operators in $\B(H)$ and let $\A=(T_{j-k})_{k,j}$.
If $\A\in {\mathcal M}_l(\l2)$ then $\Phi_\A\in \B(C(\T,H), H)$ and
$$\norm{\Phi_\A}_{\B(C(\T,H), H)}\le \norm{\A}_{{\mathcal M}_l(\l2)}.$$
\end{teor}
\begin{proof} By assumption, Theorems \ref{thm:mult_c_to_c} and \ref{prp:continuasoper} we know that for any $\sum_l S_l e^{ilt}\in P (\T,\B(H))$
\be \label{hipotesis}\sup_{t\in [-\pi,\pi]}\left\|\sum_l T_lS_le^{ilt}\right\|_{\B(H)}\le \|\A\|_{\mathcal{M}_l(\ell^2(H))} \sup_{t\in [-\pi,\pi]}\left\|\sum_l S_l e^{ilt}\right\|_{\B(H)}.\ee We need to show that for any  $\sum_l x_l e^{ilt}\in P(\T,H)$
\be \label{dual}\left\|\sum_{l}  T_lx_l\right\|\le \|\A\|_{\mathcal{M}_l(\ell^2(H))} \left\|\sum_l x_l e^{ilt}\right\|_{C(\T,H)}.
\ee
Let  $x_0\in H$ with $\|x_0\|=1$ and set  $S_l=x_0\otimes x_l$. One has that $x_l=S_l(x_0)$, $\sum_l S_l e^{ilt}\in P(\T,\B(H))$ with
$$\sup_t\left\|\sum_l x_l e^{ilt}\right\|_{H}=\sup_t\left\|\sum_l S_l e^{ilt}\right\|_{\B(H)}.$$
Using (\ref{hipotesis})  for any $x\in H$ we obtain
\ba
\left|\sum_{l} \la T_l(x_l),x\ra\right|&=& \left|\Presc{\sum_{l}  T_lS_l(x_0)}{x}\right|\\
&\le& \|x\|\left\|\sum_l T_lS_l\right\|_{\B(H)}\\
&\le&\|x\|\sup_{t\in [-\pi,\pi]}\left\|\sum_l T_lS_le^{ilt}\right\|_{\B(H)}\\
&\le&\|\A\|_{{\mathcal M}_l(\l2)} \|x\|\sup_{t\in [-\pi,\pi]}\left\|\sum_l S_l e^{ilt}\right\|_{\B(H)}\\
&\le&\|\A\|_{{\mathcal M}_l(\l2)} \|x\|\sup_{t\in [-\pi,\pi]}\left\|\sum_l x_l e^{ilt}\right\|.
\ea

This gives  (\ref{dual}) and, using density of polynomials,  shows that $\Phi_\A$ extends to a bounded linear operator from  $C(\T,H)$ into $H$ with $\|\Phi_\A\|_{\B(C(\T,H),H))}\le \|\A\|_{{\mathcal M}_l(\l2)}.$
\end{proof}

  If a Toeplitz matrix $\A$ is considered to act as a multiplier just on Toeplitz matrices, the SOT-measures actually give a description of these multipliers, as the following theorem states.

\begin{teor} \label{thm:caract_msot}
Let $(T_l)_{l\in \Z}$ be a sequence of operators in $\B(H)$ and let $\A=(T_{j-k})_{k,j}$. Then,
$\A\in \left(\Bl2_\mathcal{T},\Bl2_\mathcal{T}\right)_l$ if and only if $\Phi_\A\in \B(C(\T,H), H)$. Furthermore
$$\norm{\Phi_\A}_{\B(C(\T,H), H)}= \norm{\A}_{{\mathcal M}_l(\l2)}.$$
\end{teor}

\begin{proof}
The direct implication was given in Theorem \ref{thm1}.

Assume that  $\Phi_\A$ extends to an element in $\B(C(\T,H), H)$.  Due to Theorem \ref{thm:mult_c_to_c},
it suffices to show that $\A*{\bf B}\in \Cl2$ for any ${\bf B}\in \Cl2_\mathcal T$. Or equivalently that
\be \label{ine}
\|\A*{\bf P}\|_{\Bl2} \le \norm{\Phi_\A}_{\B(C(\T,H), H)}\|{\bf P}\|_{\Bl2}, \quad {\bf P}\in \mathcal P(\l2).
\ee
Since by Theorem \ref{prp:continuasoper}, using the identification between ${\bf B}$ and $g_{\bf{B}}$ and, taking into account that $g_{\A*\bf{B}}= g_\A*g_{\bf{B}}$ we need to prove that if $P(t)=\sum_l S_l e^{ilt}\in P(\T, \B(H))$ then
\be \label{estimate}\sup_{t\in [-\pi,\pi]}\left\|\sum_l T_lS_l e^{ilt}\right\|_{\B(H)}\le \norm{\Phi_\A}_{\B(C(\T,H), H)}\|P\|_{C(\T, \B(H))}.\ee
Now observe that
\begin{equation*}
\begin{split}
\sup_{t\in [-\pi,\pi]}\Norm{\sum_lT_lS_l e^{ilt}}_{\mathcal{B}(H)}
& =\sup_{t\in [-\pi,\pi],\norm{x}=1}\Norm{\sum_lT_l(S_l(x))e^{ilt}}=\\
&=\sup_{t\in [-\pi,\pi],\norm{x}=1}\Norm{\Phi_\A\left(\sum_lS_l(x) e^{ilt} \varphi_l\right)}\leq\\
&\leq\sup_{t\in [-\pi,\pi],\norm{x}=1}\Norm{\Phi_\A}\sup_{s\in [-\pi,\pi]}\Norm{\sum_lS_l(x) e^{il(t-s)} }\leq\\
&\leq\sup_{t\in [-\pi,\pi],\norm{x}=1}\norm{\Phi_\A}\cdot\norm{x}\cdot\sup_{s\in [-\pi,\pi]}\Norm{\sum_lS_l e^{il(t-s)} }_{\mathcal{B}(H)}=\\
&=\norm{\Phi_\A}\cdot\norm{P}_{C(\T,\B(H))}
\end{split}
\end{equation*}
This finishes the proof.
\end{proof}

\begin{coro}
\label{cor:caract_msot_2}
Let $\A=(T_{j-k})_{k,j}$ be a Toeplitz matrix. Then

$\A\in \left(\Bl2_\mathcal{T},\Bl2_\mathcal{T}\right)_r$ if and only if there exists $\mu_\A\in M_{SOT}(\mathbb{T},\mathcal{B}(H))$
such that $\widehat{\mu_\A}(l)=T_{-l}$ for all $l\in \Z$.
Furthermore, $\norm{\mu_\A}_{M_{SOT}(\T,\B(H))}=\norm{\A}_{\Ml2}.$
\end{coro}
\begin{proof}
Consider $\bold{A^*}=(T_{k-j}^*)_{k,j}$. Using the fact that
$$\A\in \left(B(\ell^2(H))_\mathcal{T},B(\ell^2(H))_\mathcal{T}\right)_r
\Leftrightarrow \A^*\in \left(B(\ell^2(H))_\mathcal{T},B(\ell^2(H))_\mathcal{T}\right)_l$$

together with a combined use of theorems \ref{thm1} and \ref{thm:caract_msot}, we have that $\A\in \left(\Bl2_\mathcal{T},\Bl2_\mathcal{T}\right)_r$ if and only if $\Phi_{\A^*}\in\mathcal{B}(C(\T,H),H)$, with $\norm{\Phi_{\A^*}}_{\B(C(\T,H), H)}= \norm{\A^*}_{{\mathcal M}_l(\l2)}$. Denote by $\nu$ the element of $\mathcal{M}(\T,\mathcal{B}(H))$ associated to $\Phi_{\A^*}$, whose coefficients are $\hat{\nu}(j-k)=T_{k-j}^*$. Invoking \cite[Proposition 3.2]{BB}, we get that $\mu_\A:=(\nu)^*\in M_{SOT}(\T,\mathcal{B}(H))$, with $\widehat{\mu_\A}(j-k)=((T_{k-j}^*))^*=T_{k-j}$ and $\|\mu_\A\|_{M_{SOT}(\T, \B(H))}=\norm{\Phi_{\A^*}}_{\B(C(\T,H), H)}=\norm{\A^*}_{{\mathcal M}_l(\l2)}=\norm{\A}_{{\mathcal M}_r(\l2)}$, which concludes the proof.
\end{proof}

\subsection{A matriceal version of the disc algebra}

Recall as in the introduction  that if $X$ is a complex Banach space, we write $\mathcal H(\D,X)$ for the space of $X$-valued holomorphic functions, $H^\infty(\D, X)$ for the Banach space of bounded analytic functions on the unit disc with values in $X$ and $A(\D,X)$ stands for the disc algebra consisting in functions $f:\D\to X$ that are holomorphic and also extend to a continuous function on the closure of $\D$, with the norm
$$\norm{f}_{H^\infty(\D,X)}=\sup\lbrace{\|f(z)\|\;\;|\; z\in \D\rbrace},$$
$$\norm{f}_{A(\D,X)}=\sup\lbrace{\|f(z)\|\;\;|\; z\in \overline{\D}\rbrace}.$$
In what follows we present a version of these spaces for matrices with entries in $\mathcal{B}(H)$.

When assuming that $\A=(T_{k,j})_{k,j}\in \mathcal U$ satisfies the condition (\ref{hip0})
we can guarantee that
 $$F_\A(z)=\sum_{l=0}^\infty \Dl z^l\in {\mathcal H}(\D, \Bl2)$$
 is a well defined holomorphic function. It follows from the definitions that
$$F_\A(re^{it})={\bf M}_{P_r}\ast f_\A(t)=\sum_{l=0}^{\infty}{\bf D}_l r^le^{ilt}.$$

Similarly if $\A= (T_{j-k})\in {\mathcal U}\cap {\mathcal T} $ then  $\tilde F_\A(z)=\sum_{l=0}^\infty T_l z^l\in {\mathcal H}(\D,\B(H))$.

\begin{teor}
Let $\A=(\Tkj)\in \mathcal U$ satisfying (\ref{hip0}). Then

(i) $\A\in \Bl2$ if and only if $F_\A\in H^\infty(\D, \Bl2)$. Moreover $\|\A\|_{\Bl2}= \|F_\A\|_{H^\infty(\D, \Bl2)}.$

(ii) $\A\in \mathcal{C}(\ell^2(H))$ if and only if $F_\A\in A(\D, \B(\ell^2(H))$.

\end{teor}

\begin{proof} (i) Using part (i) in Proposition \ref{prop:sobre_las_sigmas} for $k_n= P_{r_n}$ for a sequence $r_n$ converging to $1$ we have that
$\A\in \Bl2$ if and only if $\|F_\A\|_{H^\infty(\D, \Bl2)}=\sup_n\|P_{r_n}(\A)\|_{\Bl2}<\infty$  with equality of norms.

(ii) Using Theorem \ref{thm:caract_cont} we know that $\A\in \mathcal{C}(\ell^2(H))$ if and only if $f_\A\in C(\T, \Bl2)$.
Using now that $F_\A(re^{it})=P_r(f_\A(t))$ and invoking part (i) we have that $$\normbh{P_r(\A)-\A}=\|P_r*f_\A- f_\A\|_{C(\T,\Bl2)}$$
what gives the result.

\bigskip
\end{proof}

Similar arguments can be applied to get the following corollary.
\begin{coro}
Let $\A=(T_{j-k})\in \mathcal U\cap \mathcal T$ satisfying $$\sup_l \|T_l\|<\infty.$$

(i) $\A\in \Bl2$ if and only if $\tilde F_\A\in H^\infty(\D, \B(H))$. Moreover $\|\A\|_{\Bl2}= \|\tilde F_\A\|_{H^\infty(\D, \B(H))}.$

(ii) $\A\in \mathcal{C}(\ell^2(H))$ if and only if $\tilde F_\A\in A(\D, \B(H))$.

\end{coro}

\vspace{.1in}
\[
\begin{tabular}{lccl}
Departamento de An\'alisis Matem\'{a}tico &   \\
Universidad de Valencia &  \\
46100 Burjassot &   \\
Valencia &   \\
Spain &  &  & \\
oscar.blasco@uv.es & Ismael.Garcia-Bayona@uv.es \\
\end{tabular}
\]


\begin{thebibliography}{CoifRoch}

\bibitem{AP} Alexandrov, A. B. ; Peller, V.V. \textit{Hankel and Toeplitz-Schur multipliers}, Math.
Ann., 324 (2002), 277--327.


\bibitem{Be} Bennet, G. \textit{Schur multipliers} Duke Math. J. 44 (1977), 603–-639.

\bibitem{BB} Blasco, O., Garc\'{\i}a-Bayona, I., \textit{Schur product with operator-valued entries}, arXiv:1804.03432 [math.FA].

\bibitem{B} Blasco, O. \textit{Fourier Analysis on vector measures on locally compact abelian groups}
Revista de
la Real Academia de Ciencias Exactas, F{\'i}sicas y Naturales. Serie A. Matem{\'a}ticas,
2016, 110, 2, 519--539.

\bibitem{BG} B\"ottcher, A; Grudsky, S. \textit{Toeplitz Matrices, Asymptotic Linear
Algebra, and Functional Analysis.} Hindustan Book Agency, New
Delhi, 2000 and Birkh\"auser Verlag, Basel,Boston, Berlin, 2000.

\bibitem{DFS} Diestel, J.; Fourie, J.; Swart, J. \textit{The metric theory of tensor products. Grothendieck's R\'esum\'e Revisited} , American Math. Soc.2008.

\bibitem{DU}  Diestel,J.;  Uhl, J.J. \textit{Vector Measures}, Math. Surveys vol 15 Amer. Math. Soc. Providence (1977)


\bibitem{D} Duren, P.L.
\textit{Theory of $H^p$ spaces},
Academic Press, New York and London, 1970.






\bibitem{G}
 Garnett, J.B.
\textit{Bounded analytic functions},
Academic Press, New York and London, 1981.

    \bibitem{HNVW} Hytonen, T.;
 van Neerven, J.;
Veraar, M.;
Weis, L.  \textit{
Analysis in Banach spaces. Vol. I. Martingales and Littlewood-Paley theory.}
Series of Modern Surveys in Mathematics, 63.
Springer, Cham, 2016. xvi+614 pp.

\bibitem{Anca}Marcoci, A.N., Marcoci, L.G., \textit{A new class of linear operators on $\ell^2$ and Schur multipliers for them}, J. Funct. Spaces Appl. 5 (2007), no.2, 151-165.



\bibitem{PP}Persson, L-E.; Popa, N.
 \textit{Matrix Spaces and Schur Multipliers: Matriceal Harmonic Analysis}, NJ : World Scientific, 2014. 208 s.








\bibitem{Schur} Schur, J. \textit{Bemerkungen zur Theorie der beschr\"ankten Bilinearformen mit unendlich vielen
Verandlichen.} J. Reine Angew. Math. 140 (1911), 1–28.




\bibitem{S} Singer, I. \textit{Linear functionals on the space of continuous mappings of a compact Hausdorf
space into a Banach space} (in Russian), Rev. Roum. Math. Pures Appl. 2 (1957), 301--315.MR 20:3445


\bibitem{S2} Singer, I. \textit{Sur les applications lin\'eaires int\'egrales des espaces de fonctions continues. I,} Rev.
Roum. Math. Pures Appl. 4 (1959), 391--401. MR 22:5883

\bibitem{T} Toeplitz, O. \textit{Zur Theorie der quadratischen und bilinearen Formen
von unendlichvielen Veranderlichen.} Math. Annalen 70
(1911), 351-376.

\end{thebibliography}
\end{document}